\documentclass{amsart}

\usepackage{amsmath}
\usepackage{amsthm}
\usepackage{amsopn}
\usepackage{amssymb}
\usepackage[cmtip, all]{xy}

\usepackage{mathtools}
 
\usepackage{stmaryrd}

\theoremstyle{plain}
\newtheorem{Lem}{Lemma}[section]

\newtheorem{Thm}[Lem]{Theorem}

{\theoremstyle{definition} 
\newtheorem{Ex}[Lem]{Example}
\newtheorem{Rk}[Lem]{Remark}
\newtheorem{Def}[Lem]{Definition}}

\DeclareMathOperator*{\holim}{holim}
\DeclareMathOperator*{\hocolim}{hocolim}
\DeclareMathOperator*{\colim}{colim}

\DeclareMathOperator*{\lims}{lim^\mathit{s}}
\DeclareMathOperator*{\textprod}{\textstyle{\prod}}

\newcommand{\zig}{\addtocounter{Lem}{1}\tag{\theLem}}  
 
\def\:{\colon}

\title{Function spectra and continuous $G$-spectra}
\author{Daniel G. Davis}
\address{Department of Mathematics\\University of Louisiana 
at Lafayette\\Lafayette, LA 70504\\ U.S.A.}
\thanks{The author was partially supported by a grant 
(\#$\negthinspace \negthinspace$ LEQSF(2008-11)-RD-A-27) from the 
Louisiana Board of Regents.} 
\email{dgdavis@louisiana.edu}

\subjclass[2000]{55P42, 55P91, 55T15}

\begin{document}

\begin{abstract}
Let $G$ be a profinite group, 
$\{X_\alpha\}_\alpha$ a cofiltered diagram of discrete 
$G$-spectra, and $Z$ a spectrum with trivial $G$-action.  
We show how to define the homotopy fixed point spectrum 
$F(Z, \holim_\alpha X_\alpha)^{hG}$ and that 
when $G$ has finite virtual cohomological 
dimension ({\em vcd}$\,$), it is equivalent to 
$F(Z, \holim_\alpha \, (X_\alpha)^{hG})$. With these tools, we show 
that the $K(n)$-local Spanier-Whitehead dual 
is always a homotopy fixed point spectrum, a 
well-known Adams-type spectral sequence is actually a 
descent spectral sequence, and, for a sufficiently nice 
$k$-local profinite $G$-Galois extension $E$, with $K \vartriangleleft G$ and 
closed, the equivalence 
$(E^{h_kK})^{h_kG/K} \simeq E^{h_kG}$ (due to Behrens and the author), 
where $(-)^{h_k(-)}$ denotes 
$k$-local homotopy fixed points, can be upgraded to an 
equivalence that just uses ordinary 
({\em non-local}$\mspace{2.6mu}$) homotopy fixed points, when 
$G/K$ has finite vcd. 
\end{abstract}


\maketitle

\section{Introduction}\label{sectionone}
In this paper, all of our spectra are symmetric spectra of simplicial sets 
and we use $G$ to denote a profinite group. Also, 
as in \cite[Section 2.3]{joint}, we let $\Sigma\mathrm{Sp}_G$ be the 
category of discrete $G$-spectra. Thus, if $X \in 
\Sigma\mathrm{Sp}_G$, then, in particular, $X$ is a symmetric spectrum with a 
$G$-action and the symmetric sequence $\{X_i\}_{i \geq 0}$ of 
simplicial $G$-sets has the property that, for each $j \geq 0$, the 
action map on $j$-simplices,
\[G \times (X_i)_j \rightarrow (X_i)_j,\] is continuous, when the set 
$(X_i)_j$ is 
regarded as a discrete space.
\par
As in \cite[Section 4]{joint}, let 
\[\{X_\alpha\}_\alpha\] be a cofiltered diagram 
in $\Sigma\mathrm{Sp}_G$; thus, $\{X_\alpha\}_\alpha$ is 
a pro-discrete $G$-spectrum. Following \cite[Section 4]{joint}, 
we refer to the diagram 
$\{X_\alpha\}_\alpha$ as a 
{\em continuous $G$-spectrum} and the 
{\em $G$-homotopy fixed point spectrum} 
of the spectrum $\holim_\alpha X_\alpha$ is defined by 
\begin{equation}\label{hfps}\zig
\bigl(\holim_\alpha X_\alpha\bigr)^{hG} 
= \holim_\alpha \, (X_\alpha)^{hG}.\end{equation} 
\par
Now let $Z$ be any spectrum with trivial $G$-action 
and let $F(Z, \holim_\alpha X_\alpha)$ be the 
function spectrum. 
We can write the spectrum $Z$ as 
\begin{equation}\label{hocolim}\zig
Z \simeq \hocolim_\beta Z_\beta,\end{equation} 
a homotopy colimit of a directed system of finite 
spectra $Z_\beta$, and hence,
\begin{align*}
F(Z, \holim_\alpha X_\alpha) & \simeq \holim_\beta F(Z_\beta, \holim_\alpha 
X_\alpha) \\ 
& \simeq \holim_{\alpha, \beta} \, (X_\alpha \wedge DZ_\beta),\end{align*} 
where 
$DZ_\beta$ is the Spanier-Whitehead dual of $Z_\beta$. We regard 
each $DZ_\beta$ as having trivial $G$-action, so that each 
$X_\alpha \wedge DZ_\beta$, with the diagonal $G$-action, is a discrete 
$G$-spectrum. Thus, the diagram 
$\{X_\alpha \wedge DZ_\beta\}_{\alpha, \beta}$ 
is a 
continuous $G$-spectrum and hence, it is natural to make the following 
definition.
\begin{Def}\label{defone}
If $\{X_\alpha\}_\alpha$ is a continuous $G$-spectrum and $Z$ is a 
spectrum with trivial $G$-action, then we define
\begin{align*}
F(Z, \holim_\alpha X_\alpha)^{hG} & = \bigl(\holim_{\alpha, \beta}\,(X_\alpha 
\wedge DZ_\beta)\bigr)^{hG} \\ & = 
\holim_{\alpha, \beta}\,(X_\alpha 
\wedge DZ_\beta)^{hG},\end{align*} where the second equality follows 
immediately from Definition (\ref{hfps}).
\end{Def} 
Now suppose that $G$ has {\em finite vcd} 
(that is, ``finite virtual cohomological 
dimension"): this assumption means exactly that 
there is a natural number $m$ and an open subgroup 
$U$ of $G$ such that $H^s_c(U; M) = 0$, whenever $s >m$, for all discrete 
$U$-modules $M$. Here, $H^s_c(U; M)$ is the continuous cohomology 
of the profinite group $U$, with coefficients in $M$. 
\par
The above assumption of finite vcd, combined with 
\cite[Remark 7.16]{cts} and the fact that each $DZ_\beta$ is a finite 
spectrum, justifies the first equivalence in the 
following:
\begin{align*}
\holim_{\alpha, \beta}\,(X_\alpha 
\wedge DZ_\beta)^{hG} & \simeq \holim_{\alpha, \beta} \, 
\bigl((X_\alpha)^{hG} 
\wedge DZ_\beta\bigr) \\ & \simeq \holim_\alpha 
F\bigl(Z, (X_\alpha)^{\scriptscriptstyle{hG}}\bigr) 
\\ & 
\simeq F\bigl(Z, \bigl(\holim_\alpha 
X_\alpha\bigr)^{\negthinspace \scriptscriptstyle{hG}}\bigr).\end{align*}  
The above string of equivalences and the discussion that precedes it prove 
the following result.
\begin{Thm}\label{maintool} 
If the profinite group $G$ has finite vcd, 
$\{X_\alpha\}_\alpha$ is a continuous $G$-spectrum, and $Z$ is any 
spectrum with trivial $G$-action, then 
\[F(Z, \holim_\alpha X_\alpha)^{hG} \simeq 
F\bigl(Z, \bigl(\holim_\alpha X_\alpha\bigr)^{\negthinspace 
\scriptscriptstyle{hG}}\bigr).\]
\end{Thm}
\par
In the 
case when $G$ is finite, Theorem \ref{maintool} is well-known. Also, 
there is a quick and interesting application of this result 
to chromatic homotopy theory. To see this, we need to 
pause to introduce the main actors, along with some notation. 
\par
Let $n \geq 1$, let $p$ be a fixed prime, and 
let $E_n$ be the Lubin-Tate spectrum with 
\[\pi_\ast(E_n) = 
W(\mathbb{F}_{p^{n}})\llbracket u_1, ..., u_{n-1}\rrbracket [u^{\pm 1}],\] where 
$W(\mathbb{F}_{p^{n}})$ is the ring of Witt vectors of the field 
$\mathbb{F}_{p^{n}}$, each $u_i$ has degree zero, and the degree of $u$ is 
$-2$. Also, set \[G_n = 
S_n \rtimes \mathrm{Gal}(\mathbb{F}_{p^{n}}/\mathbb{F}_p),\] 
the extended Morava 
stabilizer group; $G_n$ is 
a profinite group of finite vcd. Finally, let $K(n)$ be the $n$th Morava 
$K$-theory spectrum, $L_{K(n)}(S^0)$ the $K(n)$-local 
sphere, and 
\[M_0 \leftarrow M_1 \leftarrow M_2 \leftarrow \cdots \leftarrow M_i 
\leftarrow \cdots\] a tower of generalized Moore spectra, each of 
which is finite, such that 
there is an equivalence 
$L_{K(n)}(S^0) \simeq \holim_i \, L_{E(n)}(M_i),$ 
where $E(n)$ is the Johnson-Wilson spectrum 
(see \cite[Section 2]{HoveyCech}).
\par
We recall that in \cite{DH}, for any closed 
subgroup $K$ of $G_n$, Devinatz and Hopkins construct a 
commutative $S$-algebra $E_n^{dhK}$ 
that behaves like a $K$-homotopy fixed point spectrum. 
(We note that instead of the notation 
$E_n^{dhK}$, \cite{DH} uses ``$E_n^{hK}$." However, we reserve the 
notation $E_n^{hK}$ for the homotopy fixed point spectrum of $E_n$ that is 
formed with 
respect to the continuous action of $K$ on $E_n$ (as in 
\cite[Definition 9.2]{cts}).) 
Following \cite{cts}, we use this construction of Devinatz and Hopkins to 
define 
\[F_n = \colim_{N \vartriangleleft_o 
G_n} E_n^{dhN},\] where the colimit is over all open normal subgroups of 
$G_n$; since each $E_n^{dhN}$ is a $(G_n/N)$-spectrum, 
$F_n$ is a discrete $G_n$-spectrum and $\{F_n \wedge M_i\}_i$ is a 
continuous $G_n$-spectrum. Then 
there is the equivalence 
\[E_n^{hG_n} = \bigl(\holim_i \, (F_n \wedge M_i)\bigr)^{hG_n} \simeq 
L_{K(n)}(S^0),\] thanks to 
\cite[Theorem 1, $\negthinspace \negthinspace$ (iii)]{DH} and 
\cite[Theorem 8.2.1]{joint}. 
\par
Now we are ready to return to Theorem \ref{maintool}: if 
$G = G_n$, $\{X_\alpha\}_\alpha$ 
is set equal to 
$\{F_n \wedge M_i\}_i$, and $Z = E_n \simeq \holim_i \, (F_n \wedge M_i)$, 
then this result -- together with Definition \ref{defone} -- 
makes precise 
and justifies the assertion 
\[F(E_n, L_{K(n)}(S^0)) \simeq F(E_n, E_n)^{hG_n},\] which occurs 
in \cite[end of Section 8.1]{Rognes}. 
More generally, if $Z$ is any spectrum with 
trivial $G_n$-action, there is the equivalence
\[F(Z, L_{K(n)}(S^0)) \simeq F(Z, E_n)^{hG_n},\] where the 
left-hand side in this equivalence, 
\[F(Z, L_{K(n)}(S^0)) \simeq F\bigl(L_{K(n)}(Z), L_{K(n)}(S^0)\bigr),\] 
is equivalent to the $K(n)$-local Spanier-Whitehead dual of the $K(n)$-local 
spectrum $L_{K(n)}(Z)$. Thus, the functional dual in the $K(n)$-local 
category is given by $G_n$-homotopy fixed points.
\begin{Rk}
As shown in \cite[Section 1; Corollary 5.3]{jhrs}, there is an 
equivalence $(F_n)^{hG_n} \simeq L_{K(n)}(S^0)$, 
and hence, for $Z$ an arbitrary spectrum with trivial 
$G_n$-action, \[F\bigl(L_{K(n)}(Z), L_{K(n)}(S^0)\bigr) 
\simeq F\bigl(Z, (F_n)^{hG_n}\bigr) \simeq F(Z, F_n)^{hG_n}.\] Thus, we 
can conclude that the functional dual of $L_{K(n)}(Z)$ 
in the $K(n)$-local category is also given by the 
$G_n$-homotopy fixed points of the function spectrum $F(Z, F_n)$, 
which, curiously, is not necessarily $K(n)$-local (for example, when $Z = S^0$, 
$F(Z, F_n) \cong F_n$ is not $K(n)$-local, by \cite[Lemma 6.7]{cts}).
\end{Rk}
\par
In addition to the above conclusions, 
Theorem \ref{maintool} is also useful for further developing 
the theory of homotopy fixed points in at least two other ways: 
it plays a role in obtaining 
Theorem \ref{iterated},  
which is a result about iterated 
homotopy fixed points for a certain type of 
profinite Galois extension; and, for $K$ any closed subgroup 
of $G_n$ and $Z$ any spectrum, we show in Theorem \ref{secondmain} 
that 
the strongly convergent Adams-type spectral sequence 
with abutment $(E_n^{dhK})^\ast(Z)$, constructed by Devinatz and 
Hopkins in \cite{DH}, 
is actually a descent spectral sequence 
for the homotopy fixed point spectrum $F(Z, E_n)^{hK}$. 
To keep this Introduction from being 
unnecessarily redundant, we defer a fuller exposition of these 
two applications to Sections \ref{sectiontwo}~and~\ref{sectionthree}.
\vspace{.1in}
\par
\noindent
\textbf{Acknowledgements.} I thank Mark Behrens for discussions 
during our collaboration on \cite{joint} that were useful for the writing of this 
paper.
\section{Iterated homotopy fixed points for profinite 
Galois extensions}\label{sectiontwo}
Our first extended 
application of the tools developed in the Introduction is to the theory of 
profinite Galois extensions; for background on these extensions, we 
refer the reader to \cite{joint} and \cite{Rognes}. 
\par
We begin by establishing some notation. As in \cite{joint}, suppose that 
$L_k(-)$, Bousfield 
localization with respect to the spectrum $k$, satisfies the equivalence 
\[L_k(-) \simeq L_M(L_T(-)),\] where $M$ is 
a finite spectrum and $T$ is smashing. Also, suppose that $A$ is a cofibrant 
commutative symmetric ring spectrum that is $k$-local. Finally, given 
a profinite group $H$, let $(-)^{h_kH}$ denote 
the right derived functor of 
the fixed points $(-)^H$, with respect to the $k$-local model 
structure on discrete $H$-spectra (see \cite[Section 6.1]{joint}): given 
a discrete $H$-spectrum $Y$, 
\[Y^{h_kH} = (Y_{f_kH})^H,\] where $Y \rightarrow Y_{f_kH}$ is a natural 
trivial 
cofibration, with $Y_{f_kH}$ fibrant, in the $k$-local model structure on 
discrete $H$-spectra (henceforth, we will say that $Y_{f_kH}$ is a 
``$k$-locally fibrant discrete $H$-spectrum").
\par
As in \cite[Section 7]{joint}, we let $E$ be a consistent profaithful 
$k$-local profinite $G$-Galois extension of $A$ of finite vcd and 
we recall that by \cite[Proposition 6.2.3]{joint}, $E$ is a discrete 
$G$-spectrum. Also, we let $K$ be 
any closed normal subgroup of $G$. By 
\cite[Proposition 7.1.4]{joint}, $(E_{f_kG})^K$ is $k$-locally 
fibrant as a discrete $(G/K)$-spectrum, and hence, since the 
$(G/K)$-equivariant 
map \[\lambda \: 
(E_{f_kG})^K \xrightarrow{\simeq_k} \bigl((E_{f_kG})^K\bigr)_{f_kG/K}\] is a 
$k$-local equivalence 
between $k$-locally fibrant discrete $(G/K)$-spectra, the induced 
map $\lambda^{G/K}$, which has the form
\[E^{\scriptscriptstyle{h_kG}} = (E_{\scriptscriptstyle{f_kG}})^
{\scriptscriptstyle{G}} = \bigl((E_{\scriptscriptstyle{f_kG}})^{\scriptscriptstyle{K}}\bigr)^{\scriptscriptstyle{G/K}} 
\xrightarrow[\scriptscriptstyle{\lambda}^{\scriptscriptstyle{G/K}}]{\simeq_k}
\bigl(\bigl((E_{\scriptscriptstyle{f_kG}})^
{\scriptscriptstyle{K}}\bigr)_{\scriptscriptstyle{f_kG/K}}\bigr)^
{\scriptscriptstyle{G/K}} = 
((E_{\scriptscriptstyle{f_kG}})^{\scriptscriptstyle{K}})^
{\scriptscriptstyle{h_kG/K}},\] 
is a $k$-local equivalence. By \cite[Proposition 6.1.7, 
$\negthinspace \negthinspace \negthinspace$ (1)]{joint}, 
the source $E^{h_kG}$ and target 
$((E_{f_kG})^K)^{h_kG/K}$ of $\lambda^{G/K}$ are $k$-local spectra. 
Therefore, 
\begin{equation}\label{pregalois}\zig
\lambda^{G/K} \: E^{h_kG} \overset{\simeq}{\longrightarrow} 
((E_{f_kG})^K)^{h_kG/K}\end{equation} is a weak equivalence of spectra. 
\par
Notice that there is a zigzag of $k$-local equivalences 
\begin{equation}\label{prezigzag}\zig
E_{f_kG} \overset{\simeq_k}{\longrightarrow} 
(E_{f_kG})_{f_kK} \overset{\simeq_k}{\longleftarrow} 
E_{f_kK}\end{equation}
of discrete $K$-spectra.
Then taking the $K$-fixed points of zigzag (\ref{prezigzag}) 
gives the zigzag
\begin{equation}\label{zigzag}\zig
(E_{f_kG})^K \overset{\simeq_k}{\longrightarrow} 
\bigl((E_{f_kG})_{f_kK}\bigr)^K \overset{\simeq}{\longleftarrow} 
(E_{f_kK})^K = E^{h_kK},\end{equation} where the 
second map is a weak equivalence of spectra, since it is the $K$-fixed 
points of a $k$-local equivalence between $k$-locally fibrant discrete 
$K$-spectra, and the first map 
is a $k$-local equivalence, due to the combination of the last conclusion 
(about the second map) and the end of 
\cite[proof of Theorem 7.1.6]{joint}.
\par
In the $k$-local model structure on discrete $(G/K)$-spectra, the 
weak equivalences are those morphisms that are 
$k$-local equivalences, and hence, though $\bigl((E_{f_kG})_{f_kK}\bigr)^K$ 
and $(E_{f_kK})^K$ do not necessarily carry (pertinent, nontrivial) $(G/K)$-actions, 
zigzag (\ref{zigzag}) makes it reasonable 
-- in the $k$-local setting -- to identify the discrete $(G/K)$-spectrum 
$(E_{f_kG})^K$ with $E^{h_kK}$, so that equivalence 
(\ref{pregalois}) can be interpreted as the equivalence
\begin{equation}\label{galois}\zig
(E^{h_kK})^{h_kG/K} \simeq E^{h_kG}.\end{equation}
\par
Equivalence (\ref{galois}) says 
that given a sufficiently nice profinite Galois extension $E$, the iterated 
{\em $k$-local} homotopy fixed point spectrum $(E^{h_kK})^{h_kG/K}$ 
can be formed, and it behaves in a natural way, in that it 
is equivalent to $E^{h_kG}$ and thereby mimics the fixed point identity
\[(E^K)^{G/K} = E^G.\] 
\par
Though equivalence (\ref{galois}) is a step forward in the theory 
of profinite Galois extensions, we would like to have such a result about 
iterated homotopy fixed points that {\em avoids} the $k$-local setting 
that is used in (\ref{galois}). This is not an easy thing to achieve: as 
explained in detail in 
\cite[Sections 1, 3, and 4]{iterated} and \cite[Section 3.6]{joint}, 
there are subtleties 
with (non-local) iterated homotopy fixed points that, 
in general, make even forming the iterated homotopy fixed point 
spectrum a difficult task. However, with $E$ as above, we are able to 
show that $L_M(E)$ is the homotopy limit of a continuous $G$-spectrum, 
so that one can form $(L_M(E))^{hK}$, and this last spectrum 
is the homotopy limit of a continuous $(G/K)$-spectrum, so that one can 
form the iterated homotopy fixed point spectrum 
$\bigl((L_M(E))^{hK}\bigr)^{hG/K}$. 
Additionally, we obtain the following result.
\begin{Thm}\label{iterated}
Let $E$ be a consistent profaithful 
$k$-local profinite $G$-Galois extension of $A$ of finite vcd. If $K$ is a 
closed normal subgroup of $G$ such that $G/K$ has finite vcd, 
then there is an equivalence
\[\bigl((L_M(E))^{hK}\bigr)^{hG/K} \simeq (L_M(E))^{hG}.\] 
\end{Thm}
\par
The above theorem shows that, as desired, 
a sufficiently nice profinite Galois 
extension does indeed satisfy a non-local version of equivalence 
(\ref{galois}), when the quotient group $G/K$ has finite vcd. The 
proof of Theorem \ref{iterated} and the justification for the 
two conclusions that immediately 
precede it are placed at the end of this paper, 
in Section~\ref{extension}.
\begin{Rk}
Under the hypotheses of Theorem \ref{iterated}, there is an equivalence
\[(L_M(E))^{hK} \simeq E^{h_kK},\]
by (\ref{uselater}) and \cite[Proposition 6.1.7, 
$\negthinspace \negthinspace \negthinspace$ 
(3)]{joint}, and the same argument shows that 
\[(L_M(E))^{hG} \simeq E^{h_kG}.\] Thus, the equivalence of 
Theorem \ref{iterated} is exactly the equivalence of (\ref{galois}) above, but 
presented without using $k$-local homotopy fixed points.
\end{Rk}
\par
In the following two examples, we describe some situations in which the 
cohomological conditions on $G$ and $G/K$ in Theorem \ref{iterated} 
are satisfied.
\begin{Ex}
If a profinite group $H$ is compact $p$-adic analytic, then it has finite vcd 
(an explanation is written out just after Lemma 2.9 in \cite{cts}) and 
any quotient group $H/L$, where $L$ is a closed normal subgroup, is 
again a compact $p$-adic analytic group, by \cite[Theorem 9.6]{dixon}. 
Thus, if $G$ is compact $p$-adic analytic, then 
it and the quotient $G/K$ automatically have finite vcd.
\end{Ex}
\begin{Ex}
Suppose that $G$ 
is a pro-$p$ group of finite cohomological dimension. If $K$ is nontrivial, 
topologically finitely generated, and free as a pro-$p$ group, then 
$G/K$ has finite vcd, by \cite[Theorem 0.2]{engler}. If $K$ is analytic 
pro-$p$ of dimension $d \geq 1$, then $G/K$ again has finite vcd, 
by \cite[Theorem 0.7]{engler}. There are more such results in 
\cite{engler}. 
\end{Ex}
\section{A comparison of spectral sequences}
\label{sectionthree}
In this section, we give another application of Theorem \ref{maintool}. We will 
be referring to $E_n$, $G_n$, the tower $\{M_i\}_i$, and $F_n$, as 
defined in the Introduction. We note that given a spectrum $X$, 
we always use $\pi_\ast(X)$ to refer to the graded abelian group of 
maps in the stable homotopy category from sphere spectra to $X$. 
\par
Now, let $K$ be any closed subgroup of $G_n$. In the Introduction, we 
mentioned that the commutative $S$-algebra $E_n^{dhK}$ behaves 
like a $K$-homotopy fixed point spectrum. As an example of this behavior, 
by 
\cite[Theorem 2,$\negthinspace$ (ii)]{DH}, for any spectrum $Z$ 
with trivial $K$-action, where $Z \simeq \hocolim_\beta Z_\beta$ (as in equivalence (\ref{hocolim}); recall that each $Z_\beta$ 
is a finite spectrum), 
there is a strongly convergent $K(n)_\ast$-local $E_n$-Adams 
spectral sequence 
\begin{equation}\label{s.s.}\zig
H^s_c\bigl(K; (E_n)^{-t}(Z)\bigr) \Rightarrow (E_n^{dhK})^{-t+s}(Z),\end{equation} 
where 
\begin{equation}\label{ctscohomology}\zig
H^s_c\bigl(K; (E_n)^{-t}(Z)\bigr) = \lim_{\beta, i} H^s_c\bigl(K; 
\pi_t(E_n \wedge M_i \wedge DZ_\beta)\bigr),\end{equation} with $K$ acting trivially on 
each $M_i$ and $DZ_\beta$. We note that, 
since $G_n$ is a compact $p$-adic analytic group, it follows 
that $K$ is also, by \cite[Theorem 9.6]{dixon}, and hence, since 
each 
abelian group $\pi_t(E_n \wedge M_i 
\wedge DZ_\beta)$ is a finite discrete 
$\mathbb{Z}_p\llbracket K \rrbracket$-module, 
$H^s_c\bigl(K; \pi_t(E_n \wedge M_i \wedge DZ_\beta)\bigr)$ is a finite abelian 
group, by \cite[Proposition 4.2.2]{Symonds}. 
\begin{Rk}
To avoid any confusion, we point out that 
in (\ref{ctscohomology}) above, we are using the 
presentation of $H^s_c\bigl(K; (E_n)^{-t}(Z)\bigr)$ as an inverse limit 
that comes from the discussion between Corollary 3.4 and Lemma 3.5 
in \cite{LHS} and this same paper's Proposition 3.6, instead of the 
presentation that is given by \cite[Remark 1.3]{DH}. These two (closely related) 
presentations are isomorphic, but the former is more suitable for our purposes.
\end{Rk} 
\par
By \cite{DH} (see \cite[Theorem 6.3, Corollary 6.5]{cts} for an 
explicit proof), there is an equivalence 
\[E_n \wedge M_i \simeq F_n \wedge M_i,\] for each $i$. 
Thus, 
\[(E_n)^{-t}(Z) \cong \pi_t\bigl(F(Z, \holim_i \, (F_n \wedge M_i))\bigr),\] 
where $\{F_n \wedge M_i\}_i$ is a continuous $K$-spectrum, and, since 
\[E_n^{dhK} \simeq E_n^{hK},\] by \cite[Theorem 8.2.1]{joint}, where 
\[E_n^{hK} = \bigl(\holim_i \,(F_n \wedge M_i)\bigr)^{hK},\] we have
\[(E_n^{dhK})^{-t+s}(Z) \cong \pi_{t-s}\bigl(F(Z, E_n^{hK})\bigr) 
\cong \pi_{t-s}\bigl(F(Z, E_n)^{hK}\bigr),\] where $F(Z, E_n)^{hK}$ 
is defined as in Definition \ref{defone} and the last 
isomorphism above is due to Theorem \ref{maintool} and the fact that 
$K$ has finite vcd (since $G_n$ has finite vcd). The 
observations in the preceding sentence 
suggest that spectral sequence (\ref{s.s.}) ought to be isomorphic 
to a descent spectral sequence that has the form
\begin{equation*}
H^s_c\bigl(K; \pi_t(F(Z, E_n))\bigr) 
\Rightarrow \pi_{t-s}\bigl(F(Z, E_n)^{hK}\bigr);
\end{equation*} the following theorem shows that 
this suggestion is, in fact, correct.
\begin{Thm}\label{secondmain}
Let $K$ be a closed subgroup of $G_n$ and let 
$Z$ be a spectrum with trivial $K$-action. 
Then the strongly convergent Adams-type spectral sequence 
\begin{equation}\label{strongss}\zig
H^s_c\bigl(K; (E_n)^{-t}(Z)\bigr) 
\Rightarrow (E_n^{dhK})^{-t+s}(Z)\end{equation} is 
isomorphic to the descent spectral sequence
\begin{equation}\label{dss}\zig
H^s_c\bigl(K; \pi_t(F(Z, E_n))\bigr) 
\Rightarrow \pi_{t-s}\bigl(F(Z, E_n)^{hK}\bigr),\end{equation} 
from the $E_2$-term 
onward.
\end{Thm}
\begin{proof}
Our first step 
is to show that the descent 
spectral sequence exists. Given a 
profinite group $H$ and a discrete abelian group $A$, we let 
$\mathrm{Map}^c(H,A)$ denote the abelian group of continuous 
functions $H \rightarrow A$. Then 
\[\lims_{\negthinspace \negthinspace 
\beta, i} \mathrm{Map}^c\bigl(K^m, \pi_q(F_n \wedge M_i 
\wedge DZ_\beta)\bigr) = 0,\] for all $s > 0$, all $m \geq 0$, and all $q \in 
\mathbb{Z}$ (this follows from \cite[Lemma 4.21, $\negthinspace 
\negthinspace \negthinspace$ (i)]{DH}, since the ``$G_n$" that is in 
\cite[Lemma 4.21, $\negthinspace 
\negthinspace$ (i) and its proof]{DH} can be changed to any profinite 
group, without affecting the validity of the argument), and therefore, 
by \cite[Section 4.6]{joint}, there is a homotopy spectral sequence that has 
the form
\[H^s_\mathrm{cts}\bigl(K; \pi_t(F(Z, E_n))\bigr) \Rightarrow 
\pi_{t-s}\bigl(F(Z,E_n)^{hK}\bigr),\] where 
$H^s_\mathrm{cts}\bigl(K; \pi_t(F(Z, E_n))\bigr)$ denotes the continuous 
cohomology of continuous cochains, with coefficients in the 
profinite $\mathbb{Z}_p\llbracket K \rrbracket$-module $\pi_t(F(Z,E_n))$. 
By \cite[Remark 1.3]{DH}, there is an isomorphism
\[H^s_\mathrm{cts}\bigl(K; \pi_t(F(Z, E_n))\bigr) \cong 
H^s_c\bigl(K; \pi_t(F(Z, E_n))\bigr),\] so that the above homotopy 
spectral sequence is the desired descent spectral sequence.
\par
Spectral sequence (\ref{strongss}) is the inverse limit $\lim_{\beta, i} 
{^\mathcal{A} \mspace{-1.5mu} E_r^{\ast, \ast}(\beta, i)}$ of 
$K(n)_\ast$-local 
$E_n$-Adams spectral sequences $^\mathcal{A} \mspace{-1.5mu} 
E_r^{\ast, \ast}(\beta, i)$ 
that have the form
\begin{equation*}
^\mathcal{A} \mspace{-1.5mu} E_2^{s,t}(\beta,i) = 
H^s_c\bigl(K; \pi_t(E_n \wedge M_i \wedge DZ_\beta)\bigr) 
\Rightarrow \pi_{t-s}(E_n^{dhK} \wedge M_i \wedge DZ_\beta).
\end{equation*} Similarly, spectral sequence (\ref{dss}) is the inverse 
limit $\lim_{\beta, i} {^\mathcal{D} \mspace{-2.5mu} E_r^{\ast, \ast}(\beta,i)}$ of 
descent spectral sequences ${^\mathcal{D} \mspace{-2.5mu} 
E_r^{\ast, \ast}(\beta,i)}$ 
that have the form
\[^\mathcal{D} \mspace{-2.5mu} E_2^{s,t}(\beta, i) = 
H^s_c\bigl(K; \pi_t(F_n \wedge M_i \wedge DZ_\beta)\bigr) \Rightarrow 
\pi_{t-s}(E_n^{hK} \wedge M_i \wedge DZ_\beta),\] where 
the abutment of spectral sequence 
$\lim_{\beta, i} {^\mathcal{D} \mspace{-2.5mu} E_r^{\ast, \ast}(\beta,i)}$ 
is identified by using the equivalence $F(Z, E_n^{hK}) \simeq F(Z,E_n)^{hK}$. 
By \cite[proof of Theorem 8.2.5]{joint}, 
for each $\beta$ and $i$, there is an isomorphism
\[{^\mathcal{A} \mspace{-2.5mu} E_r^{\ast, \ast}(\beta,i)} 
\cong {^\mathcal{D} \mspace{-2.5mu} E_r^{\ast, \ast}(\beta,i)}\] of 
spectral sequences from the $E_2$-terms onward, completing 
the proof of the theorem.
\end{proof}
\section{The proof of Theorem \ref{iterated}}\label{extension}
In this section, we use the notation that was established in 
Section~\ref{sectiontwo}.
\par
Since $M$ is a 
finite spectrum, \cite[Proposition~3.6]{Bousfieldlocal} implies that, for any 
spectrum $Y$,  
\[L_M(Y) \simeq F({^M\negthinspace S^0},Y), \]
where ${^M\negthinspace S^0} 
\rightarrow S^0$ denotes the $[M,-]_\ast$-colocalization of the sphere 
spectrum $S^0$. Then, as in the Introduction, by writing
\[^M\negthinspace S^0 \simeq \hocolim_{\beta} W_\beta,\] where the 
right-hand side is a 
homotopy colimit of a directed system $\{W_\beta\}_\beta$ of 
finite spectra, we obtain that 
\[L_M(Y) \simeq \holim_\beta \, (Y \wedge DW_\beta).\] 
\par
Now let $G$ be any profinite group (we are not assuming that $G$ has 
finite vcd), 
let $X$ be a discrete $G$-spectrum, and give 
$^M\negthinspace S^0$ and each $DW_\beta$ trivial $G$-action. 
Then $\{X \wedge DW_\beta\}_\beta$ is a continuous $G$-spectrum and 
\[L_M(X) \simeq \holim_\beta \, (X \wedge DW_\beta).\] These 
conclusions motivate 
the following definition. 
\begin{Def}\label{deftwo}
If $G$ is any profinite group and 
$X$ is a discrete $G$-spectrum, then it is natural to identify 
$L_M(X)$ with $F(^M \negthinspace S^0, X)$, and hence, we define
\[(L_M(X))^{hG} = F(^M \negthinspace S^0, X)^{hG},\] so that, by 
Definition \ref{defone},
\[(L_M(X))^{hG} =  \holim_\beta \, (X \wedge DW_\beta)^{hG}.\]
\end{Def}
\par
We are now ready to prove Theorem \ref{iterated}: we suppose that 
$E$ is a consistent profaithful $k$-local profinite 
$G$-Galois extension of $A$ that 
has finite vcd. Recall that $E$ is a discrete $G$-spectrum, so that 
by Definition \ref{deftwo}, 
for any closed subgroup $H$ of $G$,
\begin{align*}
(L_M(E))^{hH} & = F(^M\negthinspace S^0, E)^{hH} \\
& = \holim_\beta \, (E \wedge DW_\beta)^{hH}.\end{align*}
Since $G$ has finite vcd, 
$K$ does too, and hence, 
Theorem \ref{maintool} implies that
\[
(L_M(E))^{hK} \simeq F(^M \negthinspace S^0, E^{hK}) \simeq L_M(E^{hK}).
\]
\par
For the next step in our argument, we make a few recollections 
from \cite[Sections 2.4, 3.2]{joint}. Given any profinite group $P$, 
let \[\mathrm{Map}^c(P, E) = \colim_{N \vartriangleleft_o P} 
\mathrm{Map}(P/N, E) \cong \colim_{N \vartriangleleft_o P} 
\textprod_{\scriptscriptstyle{P/N}} E.\] Then $\mathrm{Map}^c(K, -)$ is a 
coaugmented comonad on the category of spectra 
and we let 
$\mathrm{Map}^c(K^\bullet, E)$ be the associated cosimplicial 
spectrum (obtained through the 
comonadic cobar construction), 
 which, in codegree $k$, satisfies the isomorphism
\[\bigl(\mathrm{Map}^c(K^\bullet,E)\bigr)^k 
\cong \mathrm{Map}^c(K^k, E).\] 
\par
By \cite[Theorem 3.2.1]{joint},
\[E^{hK} \simeq \holim_\Delta \mathrm{Map}^c(K^\bullet, E).\] By 
\cite[Remark 6.2.2]{joint}, $E$ is $T$-local, and hence, 
by \cite[Lemma 6.1.4, $\negthinspace \negthinspace \negthinspace$ (3)]{joint}, the cosimplicial spectrum 
$\mathrm{Map}^c(K^\bullet, E)$ is $T$-local in each codegree. Thus, 
the homotopy limit $\displaystyle{\holim_\Delta} \, 
\mathrm{Map}^c(K^\bullet, E)$ is 
$T$-local, implying that $E^{hK}$ is also $T$-local, so that 
\begin{equation}\label{uselater}\zig
(L_M(E))^{hK} \simeq L_M(E^{hK}) \simeq L_M(L_T(E^{hK})) \simeq 
L_k(E^{hK}).\end{equation}
By \cite[Corollary 7.1.3]{joint}, there is an equivalence 
\begin{equation}\label{moreuselater}\zig
L_k(E^{hK}) \simeq L_k((E_{fG})^K),\end{equation} where 
$E \rightarrow E_{fG}$ is a trivial cofibration, with $E_{fG}$ fibrant, in 
the model category $\Sigma \mathrm{Sp}_G$. Therefore, putting 
(\ref{uselater}) and (\ref{moreuselater}) together yields 
\[(L_M(E))^{hK} \simeq L_k((E_{fG})^K).\] 
\par
Notice that 
\[(E_{fG})^K \cong \colim_{K {<} \,{U} {<_o} G} (E_{fG})^U.\]
By the proof of \cite[Proposition 3.3.1, 
$\negthinspace \negthinspace$ (3)]{joint}, $(E_{fG})^U \simeq E^{hU}$. 
The argument above that showed that $E^{hK}$ is $T$-local also shows 
that each $E^{hU}$, and hence, each $(E_{fG})^U$, 
is $T$-local. Since $T$ is smashing, the filtered colimit 
$\displaystyle{\colim_{K < \, U <_o G}} \, (E_{fG})^U$ is $T$-local, 
implying that $(E_{fG})^K$ is too. We conclude that 
\begin{equation}\label{lastlabel}\zig
(L_M(E))^{hK} \simeq L_k((E_{fG})^K) \simeq L_M((E_{fG})^K)\end{equation} 
and we note that $(E_{fG})^K$ is a discrete 
$(G/K)$-spectrum. Therefore, thanks to (\ref{lastlabel}), 
we identify $(L_M(E))^{hK}$ with $L_M((E_{fG})^K)$, and hence,
\[\bigl((L_M(E))^{hK}\bigr)^{hG/K} = 
\bigl(L_M((E_{fG})^{K})\bigr)^{hG/K}.\] Thus, by Definition 
\ref{deftwo}, we have 
\[\bigl((L_M(E))^{hK}\bigr)^{hG/K} = 
F\bigl({^M\negthinspace S^0}, (E_{fG})^K\bigr)^{hG/K}.\] 
\par
Now suppose that $G/K$ has finite vcd. Then the proof of Theorem 
\ref{iterated} is completed by the equivalences
\begin{align*}
F\bigl({^M\negthinspace S^0}, (E_{fG})^K\bigr)^{hG/K} & \simeq 
F\Bigl({^M\negthinspace S^0}, \bigl((E_{fG})^K\bigr)^{hG/K}\Bigr) \\ 
& \simeq F\bigl({^M\negthinspace S^0}, E^{hG}\bigr) \\
& \simeq F\bigl({^M\negthinspace S^0}, E\bigr)^{hG} \\ 
& = (L_M(E))^{hG},\end{align*} where the 
first equivalence is due to Theorem \ref{maintool}, the second 
equivalence follows from \cite[Proposition 3.5.1]{joint}, and the third 
equivalence comes from another application of Theorem \ref{maintool}.

\end{document}